\documentclass[a4paper]{amsart}
\usepackage{a4wide}
\usepackage{microtype}
\usepackage{amsmath}
\usepackage{amsthm}
\usepackage{amsfonts}
\usepackage{amssymb}
\usepackage{mathrsfs}
\usepackage{tikz}
\usepackage[alphabetic,initials,nobysame,non-compressed-cites]{amsrefs}
\usepackage[hidelinks]{hyperref}
\usepackage{graphicx}
\usepackage{overpic}
\usepackage{colortbl}
\usepackage{todonotes}
\usepackage{enumitem}
\usepackage{tabu}
\usepackage{pdflscape}

\theoremstyle{plain}
\newtheorem*{theorem}{Theorem}
\newtheorem{proposition}{Proposition}
\newtheorem*{lemma}{Lemma}

\theoremstyle{definition}
\newtheorem*{example}{Example}
\newtheorem*{construction}{Construction}

\theoremstyle{remark}
\newtheorem*{remark}{Remark}

\def \ve {\mathbf e}
\def \vt {\mathbf t}
\def \vu {\mathbf u}
\def \vv {\mathbf v}
\def \vx {\mathbf x}
\def \vy {\mathbf y}

\def \RR {\mathbb R}
\def \ZZ {\mathbb Z}

\DeclareMathOperator{\supp}{Supp}
\DeclareMathOperator{\conv}{Conv}

\begin{document}

\title{Selectively Balancing Unit Vectors}
\author{Aart Blokhuis}
\email{a.blokhuis@tue.nl}
\author{Hao Chen}
\email{hao.chen.math@gmail.com}
\address{Departement of Mathematics and Computer Science, Technische Universiteit Eindhoven}


\begin{abstract}
	A set $U$ of unit vectors is selectively balancing if one can find two
	disjoint subsets $U^+$ and $U^-$, not both empty, such that the Euclidean
	distance between the sum of $U^+$ and the sum of $U^-$ is smaller than $1$.
	We prove that the minimum number of unit vectors that guarantee a selectively
	balancing set in $\RR^n$ is asymptotically $\frac{1}{2} n \log n$.
\end{abstract}

\thanks{
	H. Chen is supported by the Deutsche Forschungsgemeinschaft within the
	Research Training Group `Methods for Discrete Structures' (GRK 1408) and by
	NWO/DIAMANT grant number 613.009.031.
}
\maketitle

A set of unit vectors $U=\{\vu_1, \dots, \vu_m\}$ is said to be
\emph{selectively balancing} if there is a non-trivial linear combination $\vv =
\sum \varepsilon_i \vu_i$ with coefficients $\varepsilon_i \in \{-1, 0, 1\}$
such that the Euclidean norm $\| \vv \| < 1$.  In other words, $U$ is
selectively balancing if one can select two disjoint subsets $U^+$ and $U^-$,
not both empty, such that
\[
	\| \sum_{U^+}\vu - \sum_{U^-}\vu \| < 1.
\]
Note that the inequality must be strict for the problem to be nontrivial.
Otherwise, one could always balance $U$ by choosing the coefficients to be zero
for all but one unit vector.

The term ``balancing'' refers to the classical vector balancing problems, which
typically try to assign coefficients $\pm 1$ to vectors so that the signed sum
has a small norm.  Various norms could be considered for balancing vectors, and
different conditions can be imposed on the vectors; see e.g.~\cite{spencer1977,
spencer1981, barany1981, spencer1986, banaszczyk1993, giannopoulos1997,
banaszczyk1998, swanepoel2000}.  The coefficient $0$ is usually not considered,
despite its appearance in the powerful Partial Coloring Method
(see~\cite{beck1981, spencer1985} and~\cite[\S 4.5, 4.6]{matousek1999}).  In
this note, we try to balance unit vectors with Euclidean norms, and allow the
sign to be $0$.  In other words, one could abandon some (not all) vectors,
hence the term ``selectively''.

Let $\sigma(n)$ be the minimum integer $m$ such that any $m$ unit vectors in
$\RR^n$ are selectively balancing.  Our main result is

\begin{theorem}
	\[
		\sigma(n) \sim \frac{1}{2} n \log n.
	\]
	In other words, for any two constants $c_1 > 1/2 > c_2$, we have $c_1 n \log
	n > \sigma(n) > c_2 n \log n$ for sufficiently large $n$.
\end{theorem}

\begin{remark}
	The initial motivation for our investigation is a seemingly unrelated topic:
	the dot product representation of cube graphs.  A \emph{dot product
	representation}~\cite{fiduccia1998} of a graph $G=(V,E)$ is a map $\rho
	\colon V \to \RR^n$ such that $\langle \rho(u),\rho(v) \rangle \ge 1$ if and
	only if $uv \in E$.  It was conjectured~\cite{li2014} that the $(n+1)$-cube
	has no dot product representation in $\RR^n$, but was disproved by the second
	author~\cite{chen2014}.  Our construction could be modified to give dot
	product representations of $(c n \log n)$-cubes in $\RR^n$.  See the remark
	at the end for the general idea.
\end{remark}

For convenience, we take the base of the logarithm as $2$.  For two sets $A,B
\subset \RR^n$, $A+B$ denotes the Minkowski sum, i.e.\ $A + B = \{ \mathbf{a} +
\mathbf{b} \mid \mathbf{a} \in A, \mathbf{b} \in B\}$.  We will not distinguish
a set consisting of a single vector from the vector itself.

The proof of the theorem is presented in two propositions.

\begin{proposition}\label{prop:O}
	Let $c_1 > 1/2$ be a constant.  Then for sufficiently large $n$, any set of
	$c_1 n \log n$ unit vectors in $\RR^n$ is selectively balancing.
\end{proposition}

\begin{proof}
	Let $Q \subset \RR^n$ denote the unit cube $[-1/2, 1/2]^n$, and $Z$ be the
	zonotope generated by $U = \{ \vu_1, \dots, \vu_m \}$.  That is,
	\[
		Z=\left\{ \sum_{i=1}^m \lambda_i \vu_i \mid 0 \le \lambda_i \le 1 \right\}.
	\]
	In particular, $Z$ contains all the binary combinations of $U$, i.e.\ linear
	combinations with coefficients $0$ or $1$.

	Let $\vv_1$ and $\vv_2$ be two distinct binary linear combinations.  If the
	Euclidean distance $\| \vv_1 - \vv_2 \| < 1$, then $\vv = \vv_1 - \vv_2$ is a
	non-zero linear combination of $U$ with coefficients $-1$, $0$ or $1$, and
	$\| \vv \| < 1$, hence $U$ is selectively balancing by definition.  Our plan
	is to prove that there exist two distinct binary combinations of $U$ at
	Euclidean distance $<1$ if $m = c_1 n \log n$.

	Let $Z^+ = Z + Q$ and $Z^{++} = Z^+ + Q$.  Consider the translated unit cubes
	$\{ Q + \vt \mid \vt \in \ZZ^n \cap Z^+ \}$.  They are contained in $Z^{++}$
	with disjoint interiors, and form a covering of $Z$.  The number of the
	cubes, which is the cardinality of $\ZZ^n \cap Z^+$, is bounded from above by
	the volume of $Z^{++}$.

	A zonotope can be dissected into parallelepipeds generated by linearly
	independent subsets of its generator; see~\cite[\S 5]{shephard1974}
	and~\cite[\S~9.2]{beck2015}.  The volume of a parallelepiped generated by
	unit vectors is at most $1$.  Since $Z^{++}$ is (up to a translation) generated
	by $m+2n$ unit vectors, its volume is at most 
	\[
		{m+2n \choose n} < (e (\alpha+2))^n,
	\]
	where $\alpha=m/n$ and we have used Stirling's formula.  We then subdivide
	each unit cube into $(n+1)^{n/2}$ cubes of side length $1/\sqrt{n+1}$, and
	estimate, very generously, at most $(e (\alpha+2) \sqrt{n+1})^n$ cubes of
	side length $1/\sqrt{n+1}$ with disjoint interiors.  These cubes cover $Z$.
	Inside a cube of side length $1/\sqrt{n+1}$, the Euclidean distance between
	any two points is $<1$.
	
	By the pigeonhole principle and the discussion before, $U$ must be
	selectively balancing if $2^m > (e(\alpha+2)\sqrt{n+1})^n$.  If $m = \alpha n
	= c_1 n \log n$ with $c_1 > 1/2$, this condition is satisfied for
	sufficiently large $n$.
\end{proof}

The proof is illustrated in Figure~\ref{fig:O}.

\begin{figure}
	\includegraphics[width=.4\textwidth]{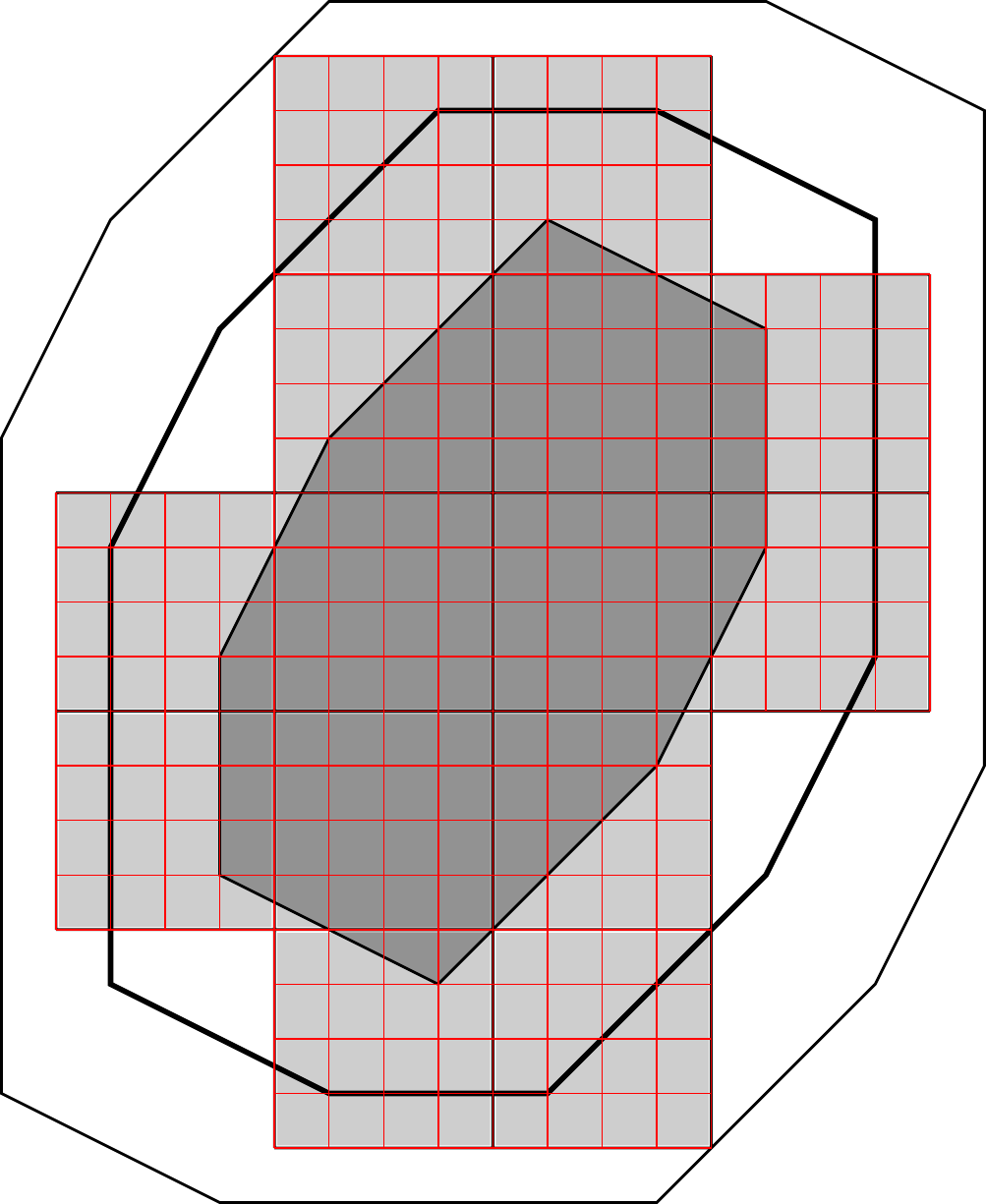}
	\caption{Proof of Proposition~\ref{prop:O}.\label{fig:O}}
\end{figure}

In order to better explain our construction for the second half of the theorem,
we would like to present an example first.

\begin{example}
	In Figure~\ref{fig:example} is an $5 \times 5$ integer lattice.  We identify
	the $25$ lattice points to the coordinates of $\RR^{25}$.  Consider two types
	of unit vectors: the first are the basis vectors; the second are half of the
	sum of the four basis vectors corresponding to the neighbors of a lattice
	point.  An example is given for each type in Figure~\ref{fig:example}, with
	the radius of the circle proportional to the component of the vector in the
	corresponding coordinate.  There are $25$ vectors of the first type, and $9$
	vectors of the second type, hence $34$ unit vectors in total.
	
	Consider a non-trivial linear combination $\vv$ of the $34$ unit vectors with
	coefficients $-1$, $0$ or $1$.  If it only involves vectors of the first
	type, we have obviously $\|\vv\| \ge 1$.  If it only involves vectors of the
	second type, one verifies that $\vv$ has absolute value $1/2$ in at least
	four coordinates, hence again $\|\vv\| \ge 1$.  If both types are involved,
	the $1/2$'s created by vectors of the second type won't be canceled by the
	integers created by basis vectors, as illustrated in
	Figure~\ref{fig:example}.  Therefore $\vv$ has absolute value $\ge 1/2$ in at
	least four coordinates, hence again $\| \vv \| \ge 1$.  We then conclude that
	these vectors are not selectively balancing.
\end{example}

\begin{figure}[hb]
	\includegraphics[width=.4\textwidth]{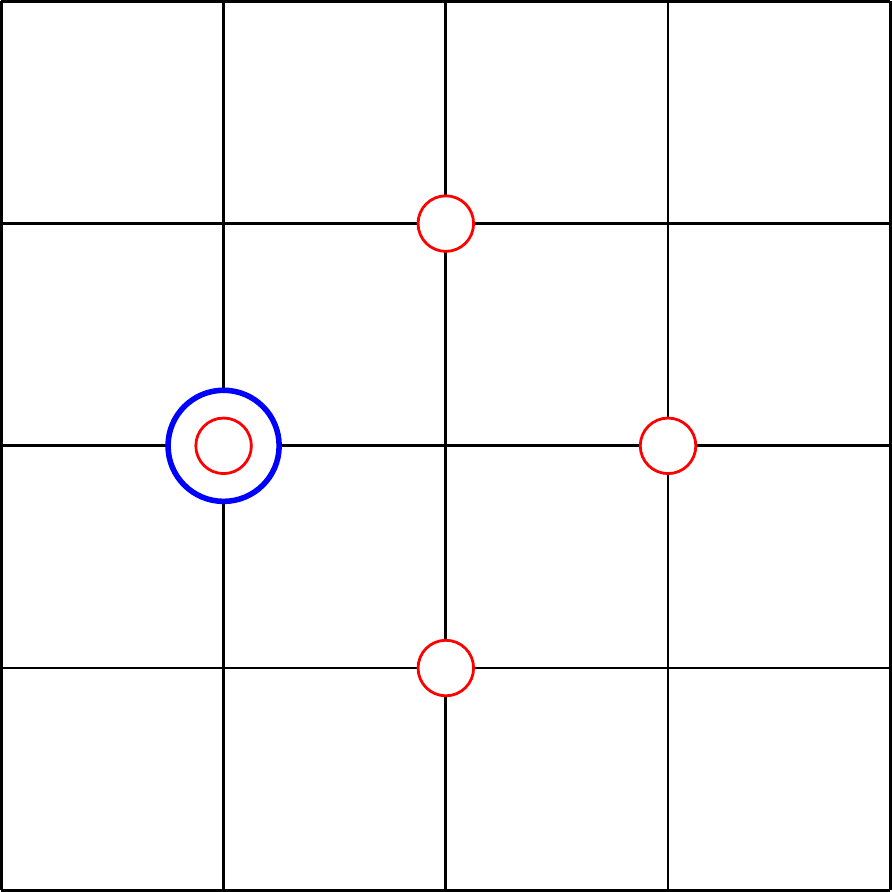}
	\caption{\label{fig:example}Construction of $34$ unit vectors in $\RR^{25}$ that are not selectively balancing.}
\end{figure}

Our construction is a generalization of the example.

We need the following lemma.  Let $C \subset \RR^d$ be a set of points in
strictly convex position, i.e.\ every point of $C$ is a vertex of the convex
hull $\conv(C)$.  Let $T$ be a finite collection of translation vectors.  Then
$C+T$ is the union of translated copies of $C$.  A point $\vx \in C + T $ is
\emph{lonely} if there is a unique $\vt \in T$ such that $\vx \in C + \vt$.

\begin{lemma}\label{lem:lonely}
	For each $\vy \in C$, there is a lonely point $\vx = \vy + \vt$.
\end{lemma}

\begin{proof}
	Since $C$ is strictly convex, there is a linear function $f \in (\RR^d)^*$
	such that $f(\vy) > f(\vy')$ for any $\vy \ne \vy' \in C$.  We then take any
	$\vx \in C + T$ that maximizes $f$.  By construction, $\vx \in C + \vt$ for
	some $\vt \in T$.  We then have $\vx = \vy + \vt$; otherwise, if $\vx = \vy'
	+ \vt$ for some $\vy' \ne \vy$, then $f(\vy + \vt) > f(\vy' + \vt)$,
	contradicting the maximality of $\vx$.  Since the translation $\vt = \vx -
	\vy$ is uniquely determined, $\vx$ is a lonely point.
\end{proof}

\begin{construction}
	Assume that $S \subset \ZZ^d$ contains at least $4^k$ integer points with the
	same Euclidean norm $R$, hence necessarily strictly convex.  Let $r=\max_{\vu
	\in S}\|\vu\|_\infty$ and $L > 2r$.  We now construct a set $U$ of
	$(k+1)(L-2r)^d$ unit vectors in $\RR^{L^d}$ that is not selectively
	balancing.

	We identify the basis vectors of $\RR^{L^d}$ to the integer points in
	$[1,L]^d$; the basis vector corresponding to $\vx \in [1,L]^d$ is denoted by
	$\ve_\vx$, and $\vv_\vx = \langle \vv, \ve_\vx \rangle$ denotes the component
	of $\vv$ in the $\vx$-coordinate.

	Let $S_0 \subset S_1 \subset S_2 \subset \dots \subset S_k = S$ be subsets of
	$S$ such that $|S_i|=4^i$ for $1 \le i \le k$.  A unit vector in our set $U$
	is labeled by an integer point $\vt \in [1+r,L-r]^d$ and an integer $i \in
	[0,k]$, and is defined by
	\[
		\vu_{\vt,i} = 2^{-i}\sum_{\vy \in S_i} \ve_{\vt+\vy}.
	\]
	Note that, by assumption, $[1+r,L-r]^d + S \subseteq [1,L]^d$, so $U \subset
	\RR^{L^d}$.  We now verify that $U$ is not selectively balancing.

	Let $\vv$ be a non-trivial linear combination of $\vu_{\vt,i}$ with coefficients
	$\varepsilon_{\vt,i} \in \{-1, 0, 1\}$.
	
	For a fixed $i \in [0,k]$, define $\vv_i=\sum \varepsilon_{\vt,i}
	\vu_{\vt,i}$ and
	\[
		\supp_i(\vv) = \{\vt \in [1+r,L-r]^d \mid \varepsilon_{\vt,i} \ne 0\}.
	\] 
	In every coordinate $\vx \in [1,L]^d$, the component $(\vv_i)_\vx$ is an
	integer multiple of $2^{-i}$.  By Lemma~\ref{lem:lonely}, for each $\vy \in
	S_i$, $S_i + \supp_i(\vv)$ has a lonely point $\vx = \vy + \vt$.  Therefore,
	in at least $4^i$ coordinates $\vx \in [1,L]^d$, we have $(\vv_i)_\vx = \pm
	2^{-i}$.

  Define
	\[
		j=\max\{i \in [0,k] \mid \supp_i(\vv) \ne \emptyset\},
	\]
	so $\vv = \vv_0 + \dots + \vv_j$.  For each $i < j$, $(\vv_i)_\vx$ is a
	multiple of $2^{-i}$, hence a multiple of $2 \times 2^{-j}$, in every
	coordinate $\vx \in [1,L]^d$.  Since $(\vv_j)_\vx = \pm 2^{-j}$ in at least
	$4^j$ coordinates, we have $| \vv_\vx | \ge 2^{-j}$ in these coordinates.  So
	$\| \vv \| \ge 1$.
\end{construction}

The second half of the theorem is proved by adjusting the parameters in the
construction.

\begin{proposition}\label{prop:Omega}
	% Let $c_2 < 1/2$ be a constant.  Then for infinitely many $n$, there are $c_2
	Let $c_2 < 1/2$ be a constant.  Then for sufficiently large $n$, there are $c_2
	n \log n$ unit vectors in $\RR^n$ that are not selectively balancing.
\end{proposition}

\begin{proof}
	Let $D=2^d$.  Consider the $(2D+1)^d$ integer points in $[-D,D]^d$.  Their
	squared Euclidean norms are at most $d D^2$.  By the pigeonhole principle,
	there is a set $S \subset [-D,D]^d$ consisting of lattice points with the
	same Euclidean norm $ R \le D \sqrt{d}$, whose cardinality
	\[
		|S| \ge \frac{(2D+1)^d-1}{d D^2} > 4^{(d^2 - d - \log d)/2}.
	\]
	Note that $r=\max_{\vu \in S}\|\vu\|_\infty \le D$.  We have constructed $m$
	unit vectors in $\RR^n$ that are not selectively balancing, where $n=L^d$ and
	\begin{align*}
		m > & \tfrac{1}{2} (d^2 - d - \log d) (L-2r)^d\\
		\ge & \tfrac{1}{2} (d^2 - d - \log d) (L-2D)^d
		\sim \tfrac{1}{2} d^2 (L-2D)^d.
	\end{align*}
	For a constant $\lambda>1$, we take $L = \lfloor D^\lambda \rfloor = \lfloor
	2^{\lambda d} \rfloor$, which is eventually bigger than $2D$.  Then $n \sim
	2^{\lambda d^2}$ and $\log n \sim \lambda d^2$, hence
	\[
		\lim_{d \to \infty} \frac{m}{n \log n} > \frac{1}{2 \lambda}.
	\]
	We conclude that, as long as $c_2 < 1/2\lambda$, there are more than $c_2 n
	\log n$ unit vectors that are not selectively balancing for sufficiently
	large integers of the form $n = \lfloor 2^{\lambda d} \rfloor^d$.

	If $n$ is sufficiently large, we can always find an integer $d$ and a
	constant $\mu \in (\sqrt\lambda, \lambda)$ such that $\lfloor 2^{\mu d}
	\rfloor^d \le n \le (\lfloor 2^{\mu d} \rfloor + 1)^d$.  Hence for any $c_2 <
	1/2\lambda$, we have $\sigma(n) > c_2 n \log n$ for sufficiently large $n$.
	This finishes the proof since we can choose $\lambda$ to be arbitrarily close
	to $1$.
\end{proof}

\begin{remark}
	We give credit to the anonymous referee for this choice of $S$, which helped
	improving the proposition.  In a preliminary version, we used another $S$,
	and proved for any $c_2 < 1/3e^2$ that $\sigma(n) > c_2 n \log n$ for
	infinitely many $n$.
\end{remark}

\begin{remark}
	In the construction, we could also replace $2$ by any integer $p>2$ and,
	correspondingly, $4$ by $p^2$.  In particular, if we take an \emph{odd}
	integer $p \ge 5$, we obtain a set of unit vectors that is \emph{strictly}
	not selectively balancing: a linear combination $\vv$ of $U$ with
	coefficients $-1$, $0$ or $1$ has Euclidean norm $1$ only if $\vv \in \pm U$.
	Our current proof for this fact is however too long and does not fit into
	this short note.

	Note that $m$ unit vectors in $\RR^n$ that are \emph{strictly} not
	selectively balancing imply a ball packing in $\RR^n$ whose tangency graph is
	a $m$-cube.  Then we conclude the following by Proposition~5
	of~\cite{kang2011}: there is a constant $c$ such that, for infinitely many
	$n$, the $(c n \log n)$-cube admits a dot product representation in $\RR^n$.
	This is actually the initial motivation of our investigation.
\end{remark}

\bibliography{References}

\end{document}